\documentclass[leqno,english]{amsart}
\usepackage[T1]{fontenc}
\usepackage[latin9]{inputenc}
\usepackage{color}
\usepackage{babel}
\usepackage{enumitem}
\usepackage{amstext}
\usepackage{amsthm}
\usepackage{setspace}
\usepackage[unicode=true,pdfusetitle,
 bookmarks=true,bookmarksnumbered=true,bookmarksopen=false,
 breaklinks=false,pdfborder={0 0 0},pdfborderstyle={},backref=page,colorlinks=true]
 {hyperref}
\hypersetup{
 citecolor=blue,linkcolor=red}

\makeatletter
\numberwithin{equation}{section}
\numberwithin{figure}{section}
\numberwithin{table}{section}
\theoremstyle{plain}
\newtheorem{thm}{\protect\theoremname}[section]
\theoremstyle{plain}
\newtheorem{conjecture}[thm]{\protect\conjecturename}
\theoremstyle{plain}
\newtheorem{prop}[thm]{\protect\propositionname}
\theoremstyle{definition}
\newtheorem{defn}[thm]{\protect\definitionname}
\theoremstyle{definition}
\newtheorem{example}[thm]{\protect\examplename}
\theoremstyle{plain}
\newtheorem{lem}[thm]{\protect\lemmaname}
\theoremstyle{remark}
\newtheorem{notation}[thm]{\protect\notationname}
\theoremstyle{plain}
\newtheorem{cor}[thm]{\protect\corollaryname}
\theoremstyle{remark}
\newtheorem{rem}[thm]{\protect\remarkname}

\usepackage{tikz,fancyhdr,enumitem}
\usetikzlibrary{scopes,calc,intersections,through,backgrounds,arrows,decorations.pathmorphing,positioning,fit,petri,decorations.markings,matrix}
\usepackage{hyperref,bookmark,dsfont,bbm}
\usepackage[all]{xy}
\hypersetup{linktocpage}


\makeatother

\providecommand{\conjecturename}{Conjecture}
\providecommand{\corollaryname}{Corollary}
\providecommand{\definitionname}{Definition}
\providecommand{\examplename}{Example}
\providecommand{\lemmaname}{Lemma}
\providecommand{\notationname}{Notation}
\providecommand{\propositionname}{Proposition}
\providecommand{\remarkname}{Remark}
\providecommand{\theoremname}{Theorem}

\begin{document}
\title[The Sticky Matroid Conjecture]{The Sticky Matroid Conjecture}
\author{Jaeho Shin}
\address{Department of Mathematical Sciences, Seoul National University, Gwanak-ro
1, Gwanak-gu, Seoul 08826, South Korea}
\email{j.shin@snu.ac.kr}
\begin{abstract}
We show Kantor's conjecture (1974) holds in rank $4$. This proves
both the sticky matroid conjecture of Poljak and Turzík (1982) and
the whole Kantor's conjecture, due to an argument of Bachem, Kern,
and Bonin, and an equivalence argument of Hochstättler and Wilhelmi,
respectively.
\end{abstract}

\maketitle

\section{Introduction}

We assume our matroids are finite throughout the paper. A matroid
is \textbf{modular} if every pair of flats of the matroid is modular,
and \textbf{hypermodular} if every pair of maximal proper flats is
modular. Modularity and hypermodularity of a matroid both are equal
in rank $3$, but different in higher ranks. Kantor made a conjecture
about the modular extensions of hypermodular matroids \cite{Kantor},
and Bonin later rephrased it \cite{Bonin}.
\begin{conjecture}[Kantor's Conjecture 1974]
Every matroid that is hypermodular is strongly embeddable in a modular
matroid.
\end{conjecture}

\begin{conjecture}[Hypermodular Matroid Conjecture 2011]
Every hypermodular matroid is a restriction of a modular matroid
of the same rank.
\end{conjecture}

A matroid $M$ is called \textbf{sticky} if for every two matroids
$M_{1}$ and $M_{2}$ with $M=M_{1}|_{E(M)}=M_{2}|_{E(M)}$ there
is a matroid $\tilde{M}$ with $M_{1}=\tilde{M}|_{E(M_{1})}$ and
$M_{2}=\tilde{M}|_{E(M_{2})}$ where $E(N)$ denotes the ground set
of $N$. In 1982, Poljak and Turzík \cite{Sticky} showed that every
modular matroid is a sticky matroid and conjectured the converse holds.
\begin{conjecture}[Sticky Matroid Conjecture 1982]
Stickiness of a matroid implies modularity. Thus, stickiness is equivalent
to modularity.
\end{conjecture}

If Kantor's conjecture holds in rank $4$, the sticky matroid conjecture
holds due to an argument of Bachem, Kern, and Bonin, \cite{BK88,Bonin}.
Recently, Hochstättler and Wilhelmi \cite{HW19} showed the sticky
matroid conjecture and Kantor's conjecture both are equivalent. In
this paper, we prove Kantor's conjecture in rank $4$ and hence both
the sticky matroid conjecture and the whole Kantor's conjecture.

\subsection*{Notations}

For a matroid $M$, unless otherwise stated, $r$ denotes its rank
function while $r(M)$ denotes its rank. For a positive integer $n$,
we denote $\left[n\right]:=\left\{ 1,\dots,n\right\} $. We denote
by $A\sqcup B$ the union of two disjoint sets $A$ and $B$.

\subsection*{Acknowledgements}

The author is deeply indebted to Thomas Zaslavsky for his invaluable
advice and carefully reading the manuscript. He is also grateful to
Joseph Bonin and Michael Wilhelmi for helpful conversations. Special
thanks to JongHae Keum for the hospitality during his visit to Korea
Institute for Advanced Study.

\section{\label{sec:Prelim}Preliminaries}

The following proposition describes the collection $\mathcal{L}(M)$
of flats of a matroid $M$ which is also called the \textbf{(geometric)
lattice} of $M$.
\begin{prop}
A subcollection $\mathcal{A}$ of the power set of a (finite) set
$S$ with $S\in\mathcal{A}$ is the lattice of a matroid on $S$ if
and only if it satisfies the following axioms.
\begin{enumerate}[label=(F\arabic*),itemsep=1pt]
\item \label{enu:(F1)}For $\ensuremath{F,L\in\mathcal{A}}$, one has $F\cap L\in\mathcal{A}$.
\item \label{enu:(F2)}For $L\in\mathcal{A}$ and $s\in S-L$, the smallest
member $F$ of $\mathcal{A}$ containing $L\cup\left\{ s\right\} $
covers $L$, that is, there is no member of $\mathcal{A}$ between
$F$ and $L$.
\end{enumerate}
The condition \ref{enu:(F2)} can be replaced with
\begin{enumerate}[label=(F2$\,'$),itemsep=1pt]
\item \label{enu:(F2')}If $L\in\mathcal{A}$ and $A_{1},\dots,A_{n}$
are all minimal members of $\mathcal{A}$ that properly contain $L$,
then $A_{1}-L,\dots,A_{n}-L$ partition $S-L$.
\end{enumerate}
\end{prop}

\begin{defn}
Let $M$ be a matroid on $E(M)=S$. For a pair $\left\{ X,Y\right\} $
of subsets of $S$, we always have $r(X)+r(Y)-r(X\cup Y)-r(X\cap Y)\ge0$,
called the submodularity of rank function $r$. The nonnegative integer
on the left is called the \textbf{modular defect }of $\left\{ X,Y\right\} $,
and $\left\{ X,Y\right\} $ is called a \textbf{modular pair} if its
modular defect is $0$. Unless otherwise specified, a (non-)modular
pair means a (non-)modular pair of flats. The \textbf{modular defect}
of $M$ is the sum of modular defects of all pairs of flats.
\end{defn}

\begin{example}
For every flat $F$, both $\left\{ F,S\right\} $ and $\{F,\overline{\emptyset}\}$
are modular pairs where $\overline{\emptyset}=\overline{\emptyset}_{M}$
denotes the set of all rank-$0$ elements (loops) of $M$, which is
a flat of rank $0$. Every pair $\left\{ F,L\right\} $ of flats with
$L\subseteq F$ is modular.
\end{example}

\begin{defn}
A flat $F$ is called \textbf{modular} if $\left\{ F,L\right\} $
for all flats $L$ is modular. The matroid $M$ is called \textbf{modular}
if every flat is modular.
\end{defn}

\begin{example}
$S$ and $\overline{\emptyset}$ are modular flats. Every rank-$1$
flat is modular.
\end{example}

\begin{defn}
For a matroid $M$ on $S$, we define the \textbf{corank} of a subset
$X\subseteq S$ to be $r(M)-r(X)$.
\end{defn}

\begin{defn}
Let $M$ be a matroid of rank $\ge3$. If every pair of corank-$1$
flats is modular, we say that $M$ is \textbf{hypermodular}.
\end{defn}

Hypermodularity is preserved under contraction.
\begin{prop}
\label{prop:inherit-HM}Let $M$ be a hypermodular matroid of rank
$\ge3$, and $F$ a flat of corank $\ge3$. Then, $M/F$ is also a
hypermodular matroid.
\end{prop}

A computation shows the following.
\begin{prop}
\label{prop:non-mod pair}Let $M$ be a loopless rank-$4$ hypermodular
matroid, and $\left\{ F,L\right\} $ a pair of flats with $r(F)\ge r(L)$.
Then, $\left\{ F,L\right\} $ is non-modular if and only if $F\cap L=\emptyset$
and either $r(F)=3$ and $r(L)=2$, or $r(F)=r(L)=2$ and $r(F\cup L)=3$.
\end{prop}

\begin{prop}[\cite{j-hope3}]
\label{prop:equiv}Let $M$ be a loopless rank-$4$ hypermodular
matroid. Then, $M$ is modular if and only if there are no two disjoint
flats of rank $3$ and $2$, if and only if no rank-$3$ flat contains
two disjoint rank-$2$ flats.
\end{prop}

The above propositions suggest what to do for the modular extensions
of rank-$4$ hypermodular matroids: reduce the number of the mentioned
non-modular pairs.

We finish this section with the following weakening of Lemma 3.5 of
\cite{j-hope3}.
\begin{lem}
\label{lem:Disjoint0}Let $M$ be a rank-$4$ loopless hypermodular
but non-modular matroid. Let $F$ and $L$ be disjoint flats of rank
$3$ and $2$, respectively, and $A_{1},\dots,A_{n}$ all rank-$3$
flats containing $L$. Then, $n\ge3$.
\end{lem}

\section{\label{sec:Extensions}Crapo's Single-element Extensions}

In this section, we review Crapo's single-element extensions of matroids.
Readers are referred to \cite{Crapo,Oxley} for further study.
\begin{defn}
A set $\mathcal{M}$ of flats of a matroid $M$ is called a \textbf{modular
cut} of $M$ if $\mathcal{M}$ satisfies the following two conditions.
\begin{enumerate}
\item If $L\in\mathcal{M}$, then $\mathcal{M}$ contains all flats $F$
of $M$ with $F\supset L$.
\item If $\left\{ F,L\right\} $ is a modular pair in $\mathcal{M}$, then
$\mathcal{M}$ contains $F\cap L$.
\end{enumerate}
\end{defn}

The empty set is said to be the \textbf{empty} modular cut. The intersection
of modular cuts is a modular cut.\smallskip{}

Let $\mathcal{T}$ be a set of corank-$2$ flats of $M$. Let $\mathrm{LS}(\mathcal{T})=\mathrm{LS}_{M}(\mathcal{T})$
denote the smallest set of corank-$1$ flats such that $\mathcal{T}'=\left\{ F\cap L:F,L\in\mathrm{LS}(\mathcal{T})\right\} $
contains $\mathcal{T}$ and every corank-$1$ flat containing a corank-$2$
flat in $\mathcal{T}'$ is a member of $\mathrm{LS}(\mathcal{T})$.
We call $\mathrm{LS}(\mathcal{T})$ the \textbf{linear subclass} \textbf{generated
by} $\mathcal{T}$, or simply a \textbf{linear subclass}. When $\mathcal{T}=\emptyset$,
we have $\mathrm{LS}(\emptyset)=\emptyset$ which is called the \textbf{empty}
linear subclass.\smallskip{}

Let $\mathrm{MC}(\mathcal{T})=\mathrm{MC}_{M}(\mathcal{T})$ denote
the union of $\left\{ E(M)\right\} $ and the set of all flats $F$
satisfying that every flat containing $F$ is an intersection of members
of $\mathrm{LS}(\mathcal{T})$, or equivalently, all flats $F$ such
that every corank-$1$ flat containing $F$ is a member of $\mathrm{LS}(\mathcal{T})$.
We say that $\mathrm{MC}(\mathcal{T})$ is \textbf{generated by }$\mathrm{LS}(\mathcal{T})$.
\begin{prop}
This set $\mathrm{MC}(\mathcal{T})$ is a nonempty modular cut.
\end{prop}

There is a one-to-one correspondence between linear subclasses and
nonempty modular cuts where for a nonempty modular cut, the set of
all corank-$1$ flats in it is its corresponding linear subclass.
Note that the modular cut generated by the empty linear subclass is
not $\emptyset$ but $\left\{ E(M)\right\} $.
\begin{defn}
A matroid $N$ on $E(N)=E(M)\sqcup\{e\}$ is called a \textbf{single-element
extension} of $M$ if $M=N|_{E(M)}$.
\end{defn}

\begin{prop}
\label{prop:Crapo}For a modular cut $\mathcal{M}$ of a matroid $M$,
there is a single-element extension $N$ with $E(N)=E(M)\sqcup\left\{ e\right\} $,
which satisfies that for all $X\subseteq E(N)$ 
\[
r_{N}(X)=\begin{cases}
r_{M}(X-\left\{ e\right\} )+1 & \text{if }e\in X\text{ and }\overline{X-\left\{ e\right\} }\notin\mathcal{M},\\
r_{M}(X-\left\{ e\right\} ) & \text{otherwise},
\end{cases}
\]
 where the closure operation is taken in $M$. To a same-rank single-element
extension of $M$ there corresponds a unique linear subclass of $M$,
and vice versa.
\end{prop}

\begin{notation}
The above single-element extension of $M$ by $\mathcal{M}$ is denoted
by $M+_{\mathcal{M}}e$.
\end{notation}

Let $\mathcal{M}$ be a modular cut of a matroid $M$. If $\mathcal{M}$
is the collection of all flats that contain a flat $F$, it is called
the \textbf{principal modular cut generated by} $F$. Then, $\mathcal{M}$
contains $\overline{\emptyset}_{M}$ if and only if $\mathcal{M}=\mathcal{L}(M)$,
that is, $\mathcal{L}(M)$ is the principal modular cut generated
by $\overline{\emptyset}_{M}$. We call $\mathcal{M}$ a \textbf{proper
modular cut} if $\overline{\emptyset}_{M}\notin\mathcal{M}$.

More generally, the smallest modular cut containing a set $\left\{ F_{1},\dots,F_{m}\right\} $
of flats is called the \textbf{modular cut generated by} $\left\{ F_{1},\dots,F_{m}\right\} $.
Note that, then, $\mathrm{MC}(\mathcal{T})$ is the modular cut generated
by $\mathcal{T}$.\smallskip{}

Next, we list a few corollaries of Proposition \ref{prop:Crapo} which
play a crucial role later in Section \ref{sec:Modular}.
\begin{cor}
\label{cor:Crapo1}For any modular cut $\mathcal{M}$ of $M$, let
$N=M+_{\mathcal{M}}e$. Then, there is a rank preserving injection
$\varphi=\varphi_{\mathcal{M}}$ from $\mathcal{L}(M)$ to $\mathcal{L}(N)$,
which satisfies 
\[
\varphi(L)=\begin{cases}
L & \text{if }L\notin\mathcal{M},\\
L\cup\left\{ e\right\}  & \text{if }L\in\mathcal{M}.
\end{cases}
\]
Further, $F\in\mathcal{L}(N)$ is not in the image of $\varphi$ if
and only if $r_{N}(F)=r_{M}(F-\left\{ e\right\} )+1$, if and only
if $F=L\cup\left\{ e\right\} $ for a proper flat $L$ that no member
of $\mathcal{M}$ covers in $\mathcal{L}(M)$.
\end{cor}

\begin{cor}
\label{cor:Crapo2}Let $\mathcal{M}$ be a modular cut of $M$ with
$N=M+_{\mathcal{M}}e$. Then, 
\begin{enumerate}
\item $\mathcal{M}=\emptyset$ if and only if $r(N)=r(M)+1$,
\item $\overline{\emptyset}_{M}\in\mathcal{M}$ if and only if $e\in\overline{\emptyset}_{N}$.
\end{enumerate}
\end{cor}

\begin{cor}
\label{cor:Crapo3}Let $\mathcal{M}$ be the non-principal modular
cut of $M$ generated by a non-modular pair $\left\{ F,L\right\} $.
Then, the modular defect of the pair $\left\{ \varphi(F),\varphi(L)\right\} $
of $M+_{\mathcal{M}}e$ is one less than that of the pair $\left\{ F,L\right\} $
of $M$.
\end{cor}

\section{\label{sec:Vamos}Vámos Line Arrangement}

\subsection*{Subspaces of a matroid}

Let $M$ be a matroid. A \textbf{subspace} of $M$ is a matroid $M/F$
for a flat $F$ of $M$, denoted by $\eta(M/F)$. We write $\eta(M)=M/\overline{\emptyset}$
by convention. We may omit $\eta$ when no confusion occurs, for instance,
when $M$ is loopless.\smallskip{}

The \textbf{(subspace)} \textbf{dimension} and the \textbf{(subspace)
codimension} of $\eta(M/F)$\textbf{ in} $\eta(M)$ are $r(M/F)-1$
and $r(F)$, respectively.\footnote{The dimension of a nonempty matroid $M$ is defined as $r(M\backslash\overline{\emptyset})-\kappa(M\backslash\overline{\emptyset})$
where $\kappa$ denotes the number of connected components, \cite[Section 4]{j-hope}.
In this paper, however, we only mean by dimension the subspace dimension.}\smallskip{}

A subspace is called a \textbf{point}, a \textbf{line}, and a \textbf{plane}
if its dimension is $0$, $1$, and $2$, respectively, and called
a \textbf{hyperplane} if its codimension is $1$. In the same sense,
we call $\eta(M)$ a \textbf{combinatorial projective space} of dimension
$r(M)-1$.\smallskip{}

If $F$ and $L$ are flats of $M$ with $L\subset F$, we say $M/L$
\textbf{contains} $M/F$, or $M/F$ is \textbf{contained in} $M/L$,
or $M/F$ is \textbf{lying on} $M/L$. Indeed, if $L$ is properly
contained in $F$, the independent-set collection of $M/L$ properly
contains that of $M/F$, that is, $L\subset F$ implies $\mathcal{I}(M/L)\supset\mathcal{I}(M/F)$.\smallskip{}

The set of all subspaces of $\eta(M)$ and the dual lattice of the
geometric lattice of the loopless matroid $M\backslash\overline{\emptyset}$
both are isomorphic as lattices.\footnote{This poset duality naturally arises by identifying the set of defining
equations of an intersection of hyperplanes in a projective space
over a field, with a flat of the corresponding matroid, which is not
just about inverting the partial order of a lattice, but has many
practical advantages. One can perform geometry over the dual poset
such as blowing-up, contracting, MMP and more. We refer readers to
\cite[Section 4]{j-hope} for more about the theory of combinatorial
hyperplane arrangements.} We say $M/F$ and $M/L$ \textbf{meet in} $M/\overline{F\cup L}$,
or the \textbf{intersection} of $M/F$ and $M/L$ is $M/\overline{F\cup L}$.
\begin{rem}
In an ambient space $\eta(M)$, the existence of a line passing through
two points fails, but the uniqueness of such a line holds.
\end{rem}

The following simple statement is the matroidal version of \textbf{Bézout's
theorem}.
\begin{prop}[Bézout's theorem for matroids]
For a matroid $M$, $m$ subspaces $M/F_{1},\dots,M/F_{m}$ meet
at a point if and only if 
\[
r(F_{1}\cup\cdots\cup F_{m})=r(M)-1.
\]
\end{prop}

\subsection*{Vámos and Anti-Vámos Line Arrangements}

Let $M$ be a matroid of rank $\ge3$ and consider an arrangement
of four lines for which no two of them are coplanar and no three of
them meet at a point. Suppose that five pairs of those four lines
have nonempty intersection. If the other pair also has nonempty intersection,
we call such an arrangement of four lines an \textbf{anti-Vámos line
arrangement}, see Figure \ref{fig:Non-Vamos}. 
\begin{figure}[th]
\begin{spacing}{0.4}
\noindent \begin{centering}
\noindent \begin{center}
\begin{tikzpicture}[font=\scriptsize]

\begin{scope}[line cap=round,scale=0.9]

 \path (0,0)--++(-30:2) coordinate (A0);
 \path (0,0)--++(90:2) coordinate (A1);
 \path (0,0)--++(150:1) coordinate (A2);
 \path (0,0)--++(210:2) coordinate (A3);
 \path (0,0)--++(270:1) coordinate (A4);

 \draw [thick] (A1)++(60:1)--(A3)--++(240:0.9) node[right]{};
 \draw [thick] (A3)++(180:1)--(A0)--++(0:1) node[above right=-2pt]{};
 \draw [thick] (A1)++(90:1)--(A4)--++(270:0.9);
 \path (A4)++(270:0.75) node[right]{};
 \draw [thick] (A2)++(150:1)--(A0)--++(-30:1) node[below right=-2pt]{};

 \foreach \x in {(A0),(A1),(A2),(A3),(A4),(0,0)}{
    \fill [black] \x circle (2.5pt);}

\end{scope}

\end{tikzpicture}
\par\end{center}
\par\end{centering}
\end{spacing}
\begin{spacing}{0.3}
\noindent \caption{\label{fig:Non-Vamos}An Anti-Vámos Line Arrangement}
\end{spacing}
\end{figure}
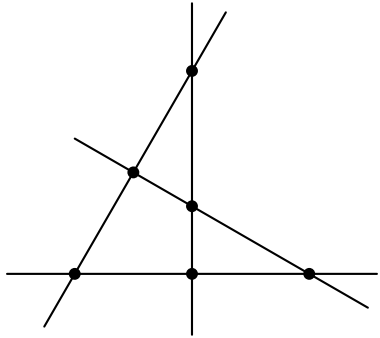
 But, this is not true in general.

Indeed, let $M$ be a Vámos matroid, which has rank $4$, on $\left[8\right]=\left\{ 1,\dots,8\right\} $
with $T_{0}=\left\{ 1,2\right\} $, $T_{1}=\left\{ 3,4\right\} $,
$L_{0}=\left\{ 5,6\right\} $, and $L_{1}=\left\{ 7,8\right\} $ such
that $T_{0}\cup T_{1}$ and $T_{i}\cup L_{j}$ with $i,j\in\left\{ 1,2\right\} $
are the five circuits of size $4$ which are also rank-$3$ flats,
then $r(L_{0}\cup L_{1})=4$. Let $t_{i}=M/T_{i}$ and $l_{i}=M/L_{i}$
for $i=0,1$, then this arrangement of four lines $\left\{ t_{0},t_{1},l_{0},l_{1}\right\} $
is a counterexample, see Figure \ref{fig:Vamos0}. We call such an
arrangement of four lines a \textbf{Vámos line arrangement}.\footnote{For a rank-$4$ matroid, the property of having no Vámos line arrangement
is equivalent to the bundle condition of \cite{BK88}.} 
\begin{figure}[th]
\begin{spacing}{0.4}
\noindent \begin{centering}
\noindent \begin{center}
\begin{tikzpicture}[font=\scriptsize]

\begin{scope}[line cap=round,scale=0.9]

 \path (0,0)--++(-30:2) coordinate (A0);
 \path (0,0)--++(90:2) coordinate (A1);
 \path (0,0)--++(150:1) coordinate (A2);
 \path (0,0)--++(210:2) coordinate (A3);
 \path (0,0)--++(270:1) coordinate (A4);

 \draw [thick] (A1)++(60:1)--(A3)--++(240:0.9) node[right]{$t_{1}$};
 \draw [thick] (A3)++(180:1)--(A0)--++(0:1) node[above right=-2pt]{$l_{1}$};
 \draw [thick] (A1)++(90:1)--(A4)--++(270:0.9);
 \path (A4)++(270:0.75) node[right]{$t_{0}$};
 \draw [thick] (A2)++(150:1)--(A0)--++(-30:1) node[below right=-2pt]{$l_{0}$};

 \foreach \x in {(A1),(A2),(A3),(A4),(0,0)}{
    \fill [black] \x circle (2.5pt);}

\end{scope}

\end{tikzpicture}
\par\end{center}
\par\end{centering}
\end{spacing}
\begin{spacing}{0.3}
\noindent \caption{\label{fig:Vamos0}A Vámos Line Arrangement}
\end{spacing}
\end{figure}
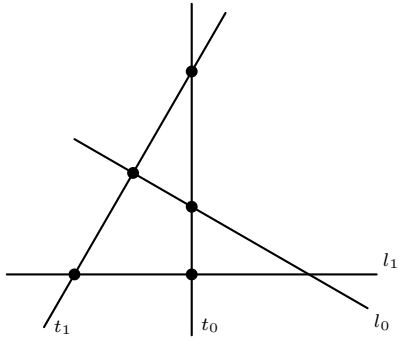

Thus, dealing with a quadruple of lines is a delicate task. But, a
triple of lines conforms to our usual geometric sense.
\begin{lem}
\label{lem:triple-lines}Let $M$ be a rank-$4$ matroid. Suppose
$M$ has three lines having empty intersection such that any two of
them meet at a point.
\begin{enumerate}
\item If a plane contains two of those three lines, it also contains the
other line.
\item If a plane contains a line of the three lines and the intersection
point of the other two, it contains all three lines.
\end{enumerate}
\end{lem}

\begin{proof}
Let $t_{i}=M/T_{i}$ for $i=1,2,3$ be those three lines where $T_{1},T_{2},T_{3}$
are rank-$2$ flats of $M$ with $r(T_{1}\cup T_{2}\cup T_{3})=4$
and $r(T_{i}\cup T_{j})=3$ for $i\neq j$.

(1) Suppose that $t_{1}$ and $t_{2}$ are contained in a plane $H$.
Then, $T_{1}\cap T_{2}$ is a rank-$1$ flat and $H=M/(T_{1}\cap T_{2})$.
Now, $\left(T_{1}\cap T_{2}\right)\cup T_{3}=\left(T_{1}\cup T_{3}\right)\cap\left(T_{2}\cup T_{3}\right)$
has rank at least $2$, and moreover at most $2$ since by submodularity
of rank function, 
\[
r(\left(T_{1}\cup T_{3}\right)\cap\left(T_{2}\cup T_{3}\right))\le r(T_{1}\cup T_{3})+r(T_{2}\cup T_{3})-r(T_{1}\cup T_{2}\cup T_{3})=2.
\]
Therefore $T_{1}\cap T_{2}\subset T_{3}$, and the plane $H$ contains
the other line $t_{3}$.

(2) Suppose that a plane $H$ contains $M/\overline{T_{1}\cup T_{2}}$
and $t_{3}$ without loss of generality, then $\overline{T_{1}\cup T_{2}}\cap T_{3}$
is a rank-$1$ flat and so $H=M/(\overline{T_{1}\cup T_{2}}\cap T_{3})$.
For each $i=1,2$, we have $T_{i}=\overline{T_{1}\cup T_{2}}\cap\overline{T_{3}\cup T_{i}}$
which contains $\overline{T_{1}\cup T_{2}}\cap T_{3}$. Thus $H$
contains both $t_{1}$ and $t_{2}$ as well as $t_{3}$. 
\end{proof}

\section{\label{sec:Modular}Kantor's Conjecture and the Sticky Matroid Conjecture}

Let $M$ be a matroid of rank $\ge3$. We may assume $M$ is loopless
for convenience. Note that extending the matroid $M$ by adding an
element that is not a coloop nor a loop is the same as adding a new
hyperplane to $M$. If $M$ is hypermodular, any two points of the
$\left(r(M)-1\right)$-dimensional projective space $M$ lie on a
(unique) line of $M$, and vice versa. Proposition \ref{prop:equiv}
has the following translation.
\begin{prop}
\label{prop:Equiv}Let $M$ be a hypermodular matroid of rank $4$.
Then, $M$ is modular if and only if any point and line of $M$ lie
on a plane, if and only if any two lines intersecting at a point lie
on a plane.
\end{prop}

\begin{notation}
For $\mathcal{X}\subseteq\mathcal{L}(M)$, we will say $M/X$ with
$X\in\mathcal{X}$ is from $\mathcal{X}$.
\end{notation}

Let $M$ be a rank-$4$ loopless hypermodular matroid that is non-modular.
Then, by Proposition \ref{prop:equiv}, the matroid $M$ has two disjoint
flats $F$ and $L$ of rank $3$ and $2$, respectively, and $\left\{ F,L\right\} $
is a non-modular pair by Proposition \ref{prop:non-mod pair}. Let
$A_{1},\dots,A_{n}$ be all rank-$3$ flats containing $L$, then
$n\ge3$ by Lemma \ref{lem:Disjoint0}, and 
\[
\mathcal{T}_{0}:=\left\{ L\right\} \cup\left\{ F\cap A_{i}:i\in\left[n\right]\right\} \neq\emptyset
\]
 is a set of corank-$2$ flats by hypermodularity of $M$; see Figure
\ref{fig:Vamos1} for the lines $M/L$ and $M/(F\cap A_{i})$ and
their intersection points $M/F$ and $M/A_{i}$. By \ref{enu:(F2')},
any two rank-$2$ flats $F\cap A_{i}$ and $F\cap A_{j}$ are disjoint
and their union has rank $r(F)=3$, and hence the pair of them is
non-modular. Thus, the modular cut generated by $\left\{ F,L\right\} $
contains a non-modular pair of rank-$2$ flats.
\begin{figure}[th]
\begin{spacing}{0.4}
\noindent \begin{centering}
\noindent \begin{center}
\begin{tikzpicture}[font=\scriptsize]

\begin{scope}[line cap=round,scale=0.9]

 \path (0,0)--++(-30:2) coordinate (A0);
 \path (0,0)--++(90:2) coordinate (A1);
 \path (0,0)--++(150:1) coordinate (A2);
 \path (0,0)--++(210:2) coordinate (A3);
 \path (0,0)--++(270:1) coordinate (A4);

 \draw [thick] (A1)++(60:1)--(A3)--++(240:0.9) node[below=-2pt]{$_{M/(F\cap A_1)}$};
 \draw [thick] (A1)++(120:1)--(A0)--++(300:0.9) node[below=-2pt]{$_{M/(F\cap A_n)}$};
 \draw [thick] (A3)++(180:1)--(A0)--++(0:1) node[below right=-4pt]{$_{M/L}$};
 \draw [thick] (A1)++(90:1)--(A4)--++(270:0.9);
 \path (A4)++(270:0.75) node[below=2pt]{$_{M/(F\cap A_i)}$};

 \foreach \x in {(A0),(A1),(A3),(A4)}{
    \fill [black] \x circle (2.5pt);}

\path (A1) node[right=1pt]{$_{M/F}$};
\path (A3)++(140:.4) node[]{$_{M/A_1}$};
\path (A4)++(30:.5) node[]{$_{M/A_i}$};
\path (A0)++(40:.4) node[]{$_{M/A_n}$};

\end{scope}

\end{tikzpicture}
\par\end{center}
\par\end{centering}
\end{spacing}
\begin{spacing}{0.3}
\noindent \centering{}\caption{\label{fig:Vamos1}The Lines from $\mathcal{T}_{0}$ and Their Intersection
Points.}
\end{spacing}
\end{figure}
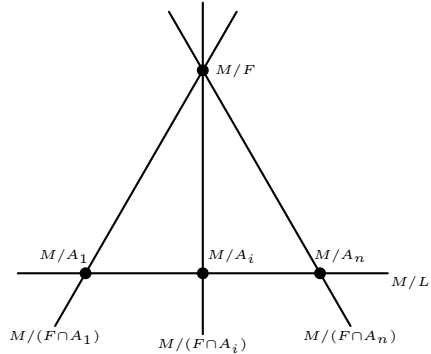

Also, $M$ has a non-modular pair $\left\{ T,T'\right\} $ of disjoint
rank-$2$ flats with $r(T\cup T')=3$ by Propositions \ref{prop:equiv}
and \ref{prop:non-mod pair}. Let $X$ and $X'$ be rank-$3$ flats
other than $\overline{T\cup T'}$ that contain $T$ and $T'$, respectively.
Then, $X\cap X'$ is a rank-$2$ flat which is disjoint from $\overline{T\cup T'}$
by Lemma \ref{lem:triple-lines}, and $\left\{ \overline{T\cup T'},X\cap X'\right\} $
is non-modular. Thus, the modular cut generated by $\left\{ T,T'\right\} $
contains a non-modular pair of flats of rank $3$ and $2$.
\begin{prop}
\label{prop:gen_nonmod}Let $M$ be a rank-$4$ loopless hypermodular
matroid that is non-modular, and $\mathcal{M}$ be a modular cut.
If $\mathcal{M}$ has a non-modular pair of rank-$2$ flats, it has
a non-modular pair of flats of rank $3$ and $2$ such that both non-modular
pairs generate the same modular cut, and vice versa.
\end{prop}

\begin{proof}
Assume the above setting. Let $\left\{ T,T'\right\} $ be a non-modular
pair of rank-$2$ flats of $\mathcal{M}$. Then, all three rank-$3$
flats $F:=\overline{T\cup T'}$, $X$ and $X'$ are contained in $\mathcal{M}$
because $\mathcal{M}$ is a modular cut, and $L:=X\cap X'$ is a rank-$2$
flat in $\mathcal{M}$ since $\left\{ X,X'\right\} $ is modular by
hypermodularity of $M$. Thus, $\left\{ F,L\right\} $ is a non-modular
pair of $\mathcal{M}$, and by construction, $\left\{ F,L\right\} $
and $\left\{ T,T'\right\} $ generate the same modular cut.

Conversely, let $\left\{ F,L\right\} $ be a non-modular pair of $\mathcal{M}$
with $r(F)=3$ and $r(L)=2$. Any two rank-$3$ flats $A_{i}$ and
$A_{j}$ contain $L$ and they are elements of $\mathcal{M}$. So,
two rank-$2$ flats $T:=F\cap A_{i}$ and $T':=F\cap A_{j}$ are contained
in $\mathcal{M}$, and $\left\{ T,T'\right\} $ is a non-modular pair
of $\mathcal{M}$ which generates the same modular cut as $\left\{ F,L\right\} $.
\end{proof}
\begin{lem}
\label{lem:non-principal}Let $M$ be a rank-$4$ loopless hypermodular
but non-modular matroid, and $\mathcal{M}$ a modular cut with a non-modular
pair. The following are equivalent.
\begin{enumerate}
\item \label{enu:proper}$\mathcal{M}$ is a proper modular cut.
\item \label{enu:non-principal}$\mathcal{M}$ is a non-principal modular
cut generated by any non-modular pair in it.
\item \label{enu:no-rank-1}$\mathcal{M}$ has no flat of rank $\le1$.
\item \label{enu:partition}The set $\mathcal{T}$ of rank-$2$ flats in
$\mathcal{M}$ is a partition of $E(M)$.
\end{enumerate}
\end{lem}

\begin{proof}
To show (\ref{enu:proper})$\Rightarrow$(\ref{enu:no-rank-1}), suppose
the proper modular cut $\mathcal{M}$ has a rank-$1$ flat $J$. Let
$\left\{ X,Y\right\} $ be a non-modular pair in $\mathcal{M}$, then
we have $X\cap J=\emptyset$ or $Y\cap J=\emptyset$, and therefore
$\emptyset\in\mathcal{M}$ since $J$ is a modular flat, which is
a contradiction. Hence, $\mathcal{M}$ has no flat of rank $\le1$.
Since (\ref{enu:no-rank-1})$\Rightarrow$(\ref{enu:proper}) is obvious,
we have (\ref{enu:proper})$\Leftrightarrow$(\ref{enu:no-rank-1}).

Suppose (\ref{enu:partition}). By hypermodularity of $M$, every
flat of rank $\le1$ is the intersection of two rank-$2$ flats. If
$J$ is a rank-$1$ flat in $\mathcal{M}$, since $\mathcal{M}$ is
a modular cut, $J$ is the intersection of two rank-$2$ flats in
$\mathcal{M}$ which is $\emptyset$, a contradiction. So, there is
no rank-$1$ flat in $\mathcal{M}$ and $\emptyset\notin\mathcal{M}$,
and hence (\ref{enu:no-rank-1}).

Suppose (\ref{enu:no-rank-1}). Any two rank-$2$ flats in $\mathcal{M}$
are disjoint. By Proposition \ref{prop:gen_nonmod}, we may assume
that $\mathcal{M}$ has a non-modular pair $\left\{ F,L\right\} $
with $r(F)=3$ and $r(L)=2$, and assume the setting of Figure \ref{fig:Vamos1},
then $E(M)$ is partitioned into $L,A_{1}-L,\dots,A_{n}-L$ by \ref{enu:(F2')}.
Also, $F$ is partitioned into rank-$2$ flats $F\cap A_{1},\dots,F\cap A_{n}$.
Similarly, for any two $i,j\in\left[n\right]$, the pair $\left\{ A_{i},F\cap A_{j}\right\} $
is a non-modular pair by Lemma \ref{lem:triple-lines} and $A_{i}$
is partitioned into rank-$2$ flats $A_{i}\cap A'_{1},\dots,A_{i}\cap A'_{n'}$
where $A'_{1},\dots,A'_{n'}$ are all rank-$3$ flats that contain
$F\cap A_{j}$. These rank-$2$ flats are flats in $\mathcal{M}$.
Thus, $E(M)$ is partitioned into rank-$2$ flats in $\mathcal{M}$,
and (\ref{enu:no-rank-1})$\Leftrightarrow$(\ref{enu:partition})
is proved.

To show (\ref{enu:proper})$\Rightarrow$(\ref{enu:non-principal}),
suppose $\mathcal{M}$ is proper. Since $\mathcal{M}$ has at least
two rank-$2$ flats and no flat of rank $\le1$ by (\ref{enu:no-rank-1}),
$\mathcal{M}$ is a non-principal modular cut. Let $\mathcal{M}'$
be the modular cut generated by the above non-modular pair $\left\{ F,L\right\} $,
and $\mathcal{T}'$ be the set of rank-$2$ flats in $\mathcal{M}'$,
then $\mathcal{T}'\subseteq\mathcal{T}$. But, both $\mathcal{T}'$
and $\mathcal{T}$ are partitions of $E(M)$ by (\ref{enu:partition})
which implies $\mathcal{T}'=\mathcal{T}$. Therefore, $\mathcal{M}'=\mathrm{MC}(\mathcal{T}')=\mathrm{MC}(\mathcal{T})=\mathcal{M}$
and $\left\{ F,L\right\} $ generates $\mathcal{M}$. Thus, any non-modular
pair in $\mathcal{M}$ generates it by Proposition \ref{prop:gen_nonmod}.

Now, (\ref{enu:non-principal})$\Rightarrow$(\ref{enu:proper}) is
obvious, and (\ref{enu:proper})$\Leftrightarrow$(\ref{enu:non-principal})
is proved. The proof is complete.
\end{proof}
\begin{cor}
\label{cor:non-principal}Let $M$ be a rank-$4$ loopless hypermodular
but non-modular matroid. Then, every non-principal modular cut is
generated by a non-modular pair.
\end{cor}

\begin{proof}
Let $\mathcal{M}$ be a non-principal modular cut, then it has a non-modular
pair of rank-$2$ flats. This pair generates $\mathcal{M}$ by Lemma
\ref{lem:non-principal}.
\end{proof}
\begin{lem}
\label{lem:SEE-0}Let $M$ be a rank-$4$ loopless hypermodular but
non-modular matroid, and $\mathcal{M}$ be a modular cut generated
by a non-modular pair. Then, $\mathcal{M}$ is a proper modular cut
if and only if $\mathcal{M}$ has no Vámos line arrangement, that
is, there are no four lines from $\mathcal{M}$ that form a Vámos
line arrangement.
\end{lem}

\begin{proof}
Suppose $\mathcal{M}$ is a proper modular cut. If $\mathcal{M}$
has a Vámos line arrangement, say $\left\{ M/T_{0},M/T_{1},M/L_{0},M/L_{1}\right\} $
with $r(L_{0}\cup L_{1})=4$, then $\left\{ L_{0},L_{1}\right\} $
is a modular pair and $\emptyset=L_{0}\cap L_{1}\in\mathcal{M}$,
a contradiction. So, $\mathcal{M}$ has no Vámos line arrangement.

Conversely, suppose $\mathcal{M}$ has no Vámos line arrangement.
We may assume $\mathcal{M}$ is generated by a non-modular pair $\left\{ F,L\right\} $
with $r(F)=3$ and $r(L)=2$ and assume the setting of Figure \ref{fig:Vamos1}.
By hypermodularity of $M$, any two points from $\mathcal{M}$ are
connected by a line of $M$, which is a line from $\mathcal{M}$ since
$\mathcal{M}$ is a modular cut. For $i=1,2,\dots$ let $\left\{ M/T:T\in\mathcal{T}_{i}\right\} $
be the set of all lines connecting two distinct points on the lines
of $\left\{ M/T:T\in\mathcal{T}_{i-1}\right\} $. Then, we obtain
an ascending chain of sets 
\[
\mathcal{T}_{0}\subset\mathcal{T}_{1}\subset\mathcal{T}_{2}\subset\cdots
\]
 which stabilizes with $\mathcal{T}=\mathcal{T}_{i}$ for some $i$
where $\mathcal{T}$ is the set of all rank-$2$ flats in $\mathcal{M}$
and $\mathcal{M}=\mathrm{MC}(\mathcal{T})=\mathrm{MC}(\mathcal{T}_{0})$.

We show $\mathcal{T}=\mathcal{T}_{1}$. Let $M/T$ be any line from
$\mathcal{T}_{1}-\mathcal{T}_{0}$, then it intersects $M/L$. Indeed,
suppose not, then $M/T$ intersects at least two lines from $\mathcal{T}_{0}$
other than $M/L$, and these two lines plus $M/T$ and $M/L$ form
an Anti-Vámos line arrangement, and $M/T$ and $M/L$ meet at a point,
a contradiction. Similarly, $M/T$ intersects every line from $\mathcal{T}_{0}$,
and further two lines from $\mathcal{T}_{1}$ intersect at a point
on a line from $\mathcal{T}_{0}$. This implies that all points on
lines from $\mathcal{T}_{1}$ are points on lines from $\mathcal{T}_{0}$,
and $\mathcal{T}=\mathcal{T}_{1}$.

Then, for any pair $\left\{ T,T'\right\} $ of two rank-$2$ flats
in $\mathcal{T}$, the flat $\overline{T\cup T'}$ is a rank-$3$
flat in $\mathcal{M}$, and $T\cap T'=\emptyset$ by Lemma \ref{lem:triple-lines}.
Thus, $\mathcal{T}$ is a disjoint collection, and $\mathcal{M}$
has no flat of rank $\le1$. Therefore $\mathcal{M}$ is a proper
modular cut by Lemma \ref{lem:non-principal}.
\end{proof}
\begin{lem}
\label{lem:SEE-1}Let $M$ be a rank-$4$ loopless hypermodular matroid
that is non-modular, and $\mathcal{M}$ a non-principal modular cut.
Then, $M+_{\mathcal{M}}e$ is a rank-$4$ loopless hypermodular matroid
whose modular defect is less than that of $M$, and 
\begin{equation}
\mathcal{L}(M+_{\mathcal{M}}e)=\varphi(\mathcal{L}(M))\sqcup\left\{ e\right\} .\label{eq:lattice}
\end{equation}
\end{lem}

\begin{proof}
By Corollary \ref{cor:Crapo2}, $M+_{\mathcal{M}}e$ is a rank-$4$
loopless matroid. By Corollary \ref{cor:non-principal}, the non-principal
modular cut $\mathcal{M}$ is generated by a non-modular pair.

Then, $\emptyset$ is only proper flat that no member of $\mathcal{M}$
covers in $\mathcal{L}(M)$. Indeed, every rank-$1$ flat is covered
by some rank-$2$ flat in the set $\mathcal{T}$ of rank-$2$ flats
in $\mathcal{M}$ since $\mathcal{T}$ is a partition of $E(M)$ by
Lemma \ref{lem:non-principal}. Every rank-$2$ flat in $\mathcal{T}$
is covered by a rank-$3$ flat in $\mathcal{M}$ since $\mathcal{M}$
is a modular cut, and every rank-$2$ flat not in $\mathcal{T}$ intersects
two rank-$2$ flats $T,T'\in\mathcal{T}$ and is covered by a rank-$3$
flat $\overline{T\cup T'}\in\mathcal{M}$. Every rank-$3$ flat is
covered by $E(M)\in\mathcal{M}$. Hence, we have (\ref{eq:lattice})
by Corollary \ref{cor:Crapo1}.

Then, the lattice structures of $\mathcal{L}(M+_{\mathcal{M}}e)$
and $\mathcal{L}(M)$ restricted to the elements of rank $\ge2$ are
the same, which implies that $M+_{\mathcal{M}}e$ is also hypermodular.
By Corollary \ref{cor:Crapo3}, the modular defect of $M+_{\mathcal{M}}e$
is less than that of $M$.
\end{proof}
\begin{thm}
\label{thm:Vamos}Every rank-$4$ hypermodular loopless matroid $M$
having no Vámos line arrangement is a restriction of a rank-$4$ modular
loopless matroid $\tilde{M}$ with $|E(\tilde{M})|-\left|E(M)\right|$
being the number of non-principal modular cuts of $M$.
\end{thm}

\begin{proof}
If $M$ is modular, nothing to prove. So, we may assume $M$ is non-modular.
By assumption, every modular cut generated by a non-modular pair has
no Vámos line arrangement and is proper and non-principal by Lemmas
\ref{lem:SEE-0} and \ref{lem:non-principal}. Thus, a modular cut
of $M$ is non-principal if and only if it is generated by a non-modular
pair by Corollary \ref{cor:non-principal}.

Let $M_{0}:=M$. By Proposition \ref{prop:equiv}, there is a non-principal
modular cut $\mathcal{M}_{0}$ and by Lemma \ref{lem:SEE-1}, the
extension $M_{1}:=M_{0}+_{\mathcal{M}_{0}}e_{1}$ is a rank-$4$ hypermodular
loopless matroid whose modular defect is less than that of $M$, which
has no Vámos line arrangement. So, we can repeat this extension process
and obtain a sequence of rank-$4$ hypermodular loopless matroids
$M_{i}$ for $i=1,2,\dots$ with non-principal modular cuts $\mathcal{M}_{i-1}$
of $M_{i-1}$ such that 
\[
M_{i}=M_{i-1}+_{\mathcal{M}_{i-1}}e_{i}.
\]
The sequence of modular defects of $M_{i}$ for $i=1,2,\dots$ is
a strictly decreasing sequence of nonnegative integers, and hence
the extension process must terminate. Let $\tilde{M}$ be the terminal
matroid, then $\tilde{M}$ is a rank-$4$ loopless matroid which is
a modular extension of $M$ since its modular defect is $0$.

In each intermediate matroid $M_{i-1}$, let $\left\{ X,Y\right\} $
be a non-modular pair of $M_{i-1}$, then its modular defect does
not exceed $1$ because $M_{i-1}$ is a hypermodular matroid of rank
$4$. By Corollaries \ref{cor:Crapo3} and \ref{cor:Crapo1}, the
pair $\left\{ \varphi_{\mathcal{M}_{i-1}}(X),\varphi_{\mathcal{M}_{i-1}}(Y)\right\} $
of $M_{i}$ is modular if $X,Y\in\mathcal{M}_{i-1}$, and non-modular
otherwise. By (\ref{eq:lattice}), every non-modular pair of $M_{i}$
is the image of a non-modular pair of $M_{i-1}$ via $\varphi_{\mathcal{M}_{i-1}}$
and that of $M$ via $\varphi_{\mathcal{M}_{i-1}}\circ\cdots\circ\varphi_{\mathcal{M}_{0}}$.
Then, every non-principal modular cut of $M_{i}$ is the image of
a non-principal modular cut of $M$ via $\varphi_{\mathcal{M}_{i-1}}\circ\cdots\circ\varphi_{\mathcal{M}_{0}}$,
which implies that $|E(\tilde{M})|-\left|E(M)\right|$ is the number
of non-principal modular cuts of $M$.
\end{proof}
Every rank-$3$ matroid satisfies the property of having no Vámos
line arrangement. For rank-$4$ matroids, that  property no longer
holds, but hypermodularity prevents the presence of Vámos line arrangement.
\begin{lem}
\label{lem:Vamos}Every rank-$4$ hypermodular matroid has no Vámos
line arrangement.
\end{lem}

\begin{proof}
Let $M$ be a rank-$4$ hypermodular matroid. We may assume $M$ is
loopless. To prove by contradiction, suppose $M$ has a Vámos line
arrangement, say $\left\{ t_{0},t_{1},l_{0},l_{1}\right\} $ of Figure
\ref{fig:Vamos0}. Let $P_{0}$ be the intersection point of $t_{1}$
and $l_{1}$, see Figure \ref{fig:Vamos3}.

There are two planes $H_{1}$ and $H_{2}$ containing $l_{1}$ which
do not contain $l_{0}$ by Lemma \ref{lem:triple-lines}. By Bézout's
theorem, they intersect $l_{0}$ at two distinct points, each of which
is connected to the point $P_{0}$ by a line due to hypermodularity
of $M$; let $l_{2}$ and $l_{3}$ be these connecting lines contained
in $H_{2}$ and $H_{1}$, respectively. Similarly, let $Y_{1}$ and
$Y_{2}$ be two planes containing $l_{0}$, then they intersect $l_{1}$
at two distinct points. Thus, we have four points, six lines connecting
two of them, and four planes whose union contains a ``tetrahedron''.

Let $H_{3}\neq H_{1}$ be another plane containing $l_{3}$, then
it does not contain $l_{2}$ by Lemma \ref{lem:triple-lines}. Similarly,
$H_{3}$ contains a line neither of the Vámos line arrangement nor
of the ``tetrahedron''. We may write $H_{3}=\eta(M')$ with $M'=M/\overline{\{h\}}$
for some $h\in E(M)$ where $\overline{\{h\}}$ is a rank-$1$ flat
of $M$, and $M'$ is a rank-$3$ loopless matroid which is hypermodular
by Proposition \ref{prop:inherit-HM}.\smallskip{}

By contracting $M$ over $\overline{\{h\}}$, the four planes $H_{1}$,
$H_{2}$, $Y_{1}$, and $Y_{2}$ of $M$ that contain four ``facets''
of the ``tetrahedron'' are transformed into four lines of $M'$,
see Figure \ref{fig:Vamos4}, while two lines $H_{1}\cap Y_{1}$ and
$H_{2}\cap Y_{2}$ of $M$ are transformed into two points of $M'$.
These two points are connected by a line of $M'$ by hypermodularity
of $M'$, which implies that the two lines $H_{1}\cap Y_{1}$ and
$H_{2}\cap Y_{2}$ of $M$ are contained in a plane of $M$, a contradiction.
Thus, we conclude that $M$ has no Vámos line arrangement. The proof
is complete.
\end{proof}
\begin{figure}[th]
\begin{spacing}{0.4}
\noindent \begin{centering}
\noindent \begin{center}
\begin{tikzpicture}[font=\scriptsize]

\begin{scope}[line cap=round,scale=1]

 \path [name path=line9] (0,1.732)--++(-60:4);
 \path [name path=line0] (0,0)--(10.5,0) node[above]{$l_{0}$};

 \foreach \x [count=\xi] in {A1,A2,A3,A4}{
    \path (2+2*\xi,1.732) coordinate(\x);
}

 \path (A1)--++(-.756*.6,-.655*.6)coordinate(A1);
 \path (A3)--++(-.918*.8,-.397*.8)coordinate(A3);
 \path (A4)--++(-.945*.5,-.327*.5)coordinate(A4);

 \path [name path=line1] (0,-1.732)--(A1);
 \path [name path=line2] (0,-1.732)--(A2);
 \path [name path=line3] (0,-1.732)--(A3);
 \path [name path=line4] (0,-1.732)--(A4);

 \path [name intersections={of=line1 and line0, by=p1}] (p1);
 \path [name intersections={of=line2 and line9, by=p2}] (p2);
 \path [name intersections={of=line3 and line0, by=p3}] (p3);
 \path [name intersections={of=line4 and line9, by=p4}] (p4);

 \path (3,0)++(210:1)coordinate(q1) (3,0)++(30:1)coordinate(q3);

 \foreach \x/\y/\z [count=\xi] in {gray!25/2/.7,blue!20/3/.5,gray!25/4/.7}{
    \fill [\x,opacity=\z,draw=none](A\xi)--(0,-1.732)--(A\y)--cycle;}

 \draw [thick] (0,-1.732)--(0,1.732)--++(-60:4) (0,0)--(2,0) (5,0)--(10.5,0);
 \draw [thick,opacity=.4] (2,0)--(5,0);

 \fill [opacity=.5,green!70] (p1)--(q3)--(p3)--(q1)--cycle;

 \path [black!80](3.5,0) node[above=-1.5pt]{$Y_{1}$};
 \path [black!80](2.5,0) node[below=-1pt]{$Y_{2}$};

 \foreach \x/\y in {1/2,2/1,3/3}{
    \draw [thick] (0,-1.732)--(A\x)node[above right=-4pt]{$l_{\y}$};}
    \draw [thick] (0,-1.732)--(A4);

 \foreach \x/\y in {1/red,3/blue}{
 \draw [very thick,\y] (p\x)--(q\x);
 }

 \foreach \x in {(0,-1.732),(0,0),(0,1.732),(1,0),(p1),(p2),(p3),(p4),(q1),(q3)}{
    \fill [black]\x circle (2pt);}

 \path (4.05,1) node{$H_2$};
 \path (5.2,1) node{$H_1$};
 \path (7.2,1) node{$H_3$};

 \path (0,-1.732) node[below left=-4pt]{$_{P_0}$};
 \path (p2) node[above=-1pt]{};
 \path (p3) node[above right=-3.5pt]{};
 \path (p4) node[above=-1pt]{};

 \path (0.5,0.8) node[above right=-3pt]{$t_{0}$};
 \path (0,0.8) node[above left=-3pt]{$t_{1}$};

\end{scope}

\end{tikzpicture}
\par\end{center}
\par\end{centering}
\end{spacing}
\begin{spacing}{0.3}
\noindent \centering{}\caption{\label{fig:Vamos3}Hypermodularity with the Vámos Line Arrangement.}
\end{spacing}
\end{figure}
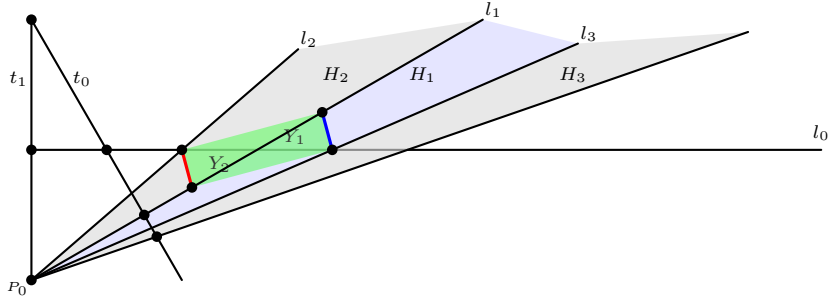
\begin{figure}[th]
\begin{spacing}{0.4}
\noindent \begin{centering}
\noindent \begin{center}
\begin{tikzpicture}[font=\scriptsize]

\begin{scope}[line cap=round,scale=1]

 \path (0,-1.732)coordinate(p0)
       (6.646,0)coordinate(p8);

 \path [name path=line2] (p0)--++(2*1.6,1.732*1.6);
 \path [name path=line7](p8)--++(-2.78*1.5,.5*1.5);

 \path [name intersections={of=line2 and line7, by=p9}] (p9);

 \path (2,0)coordinate(p1)
       (4,0)coordinate(p3);

 \path (3,0)++(210:1)coordinate(q1)
       (3,0)++(30:1)coordinate(q3);

 \draw [thick] (p1)--(p8)--(p9)--(p0)--(q3);
 \draw [very thick,green!75!black] (p1)--(q3);
 
 \foreach \x/\y in {(p1)/red,(p8)/black,(p9)/black,(p0)/black,(q3)/blue,(3,0)/black}{
    \fill [\y]\x circle (2pt);}

 \path (1.35,-0.45) node[above left=-7pt]{$_{H_2\cap H_3}$};
 \path (1.5,-0.866) node[below right=-5pt]{$_{{H_1\cap H_3}}$};
 \path (4.823,0) node[below=-1pt]{$_{{Y_2\cap H_3}}$};
 \path (4.823,0.433) node[above right=-7pt]{$_{{Y_1\cap H_3}}$};

\end{scope}

\end{tikzpicture}
\par\end{center}
\par\end{centering}
\end{spacing}
\begin{spacing}{0.3}
\noindent \centering{}\caption{\label{fig:Vamos4}An Anti-Vámos Line Arrangement in the Plane $H_{3}$.}
\end{spacing}
\end{figure}
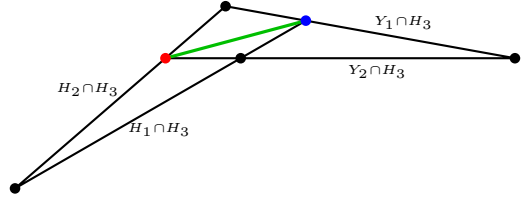

\subsection*{Conclusion}
\begin{thm}
\label{thm:Kantor-rk-4}Kantor's conjecture holds in rank $4$.
\end{thm}

\begin{proof}
Let $M$ be an arbitrary rank-$4$ hypermodular matroid. The loopless
matroid $M\backslash\overline{\emptyset}_{M}=M/\overline{\emptyset}_{M}$
is also a rank-$4$ hypermodular matroid, and has no Vámos line arrangement
by Lemma \ref{lem:Vamos}, and hence is a restriction of a rank-$4$
modular loopless matroid, say $N$, by Theorem \ref{thm:Vamos}. Then,
$N\oplus\overline{\emptyset}_{M}$ is a rank-$4$ modular matroid
whose restriction to $E(M)$ is $M$, and the proof is done.
\end{proof}
\begin{thm}
Both the sticky matroid conjecture and Kantor's conjecture hold.
\end{thm}

\begin{proof}
Theorem \ref{thm:Kantor-rk-4} implies that the sticky matroid conjecture
holds \cite{BK88,Bonin}. Now that the sticky matroid conjecture and
Kantor's conjecture both are equivalent \cite{HW19}, it follows that
the whole Kantor's conjecture holds.
\end{proof}

\end{document}